\documentclass[11pt,reqno]{amsart}
\usepackage{color}

\usepackage{amsfonts,amscd}
\usepackage{amssymb}
\usepackage{url}
\usepackage[english]{babel}

\theoremstyle{plain}
\newtheorem{theorem}                 {Theorem}      [section]

\newtheorem{corollary}    [theorem]  {Corollary}
\newtheorem{lemma}        [theorem]  {Lemma}
\newtheorem{proposition}  [theorem]  {Proposition}

\theoremstyle{definition}

\newtheorem{definition}   [theorem]  {Definition}

\newtheorem{example}      [theorem]  {Example}

\newtheorem{remark}       [theorem]  {Remark}

\numberwithin{equation}{section}

\def \cn{{\mathbb C}}
\def \hn{{\mathbb H}}
\def \H{{\mathbb H}}

\def \rn{{\mathbb R}}
\def \sn{{\mathbb S}}

\def \B{\mathcal B}
\def \E{\mathcal E}

\def \H{\mathcal H}

\def \M{\mathcal M}

\def\nab#1#2{\hbox{$\nabla$\kern -.3em\lower 1.0 ex
    \hbox{$#1$}\kern -.1 em {$#2$}}}

\def \Re{\mathfrak R\mathfrak e}

\def \g{\mathfrak{g}}

\def \GLR#1{\text{\bf GL}_{#1}(\rn)}

\def \glr#1{\mathfrak{gl}_{#1}(\rn)}
\def \GLC#1{\text{\bf GL}_{#1}(\cn)}
\def \glc#1{\mathfrak{gl}_{#1}(\cn)}
\def \GLH#1{\text{\bf GL}_{#1}(\hn)}

\def \SL2{\widetilde{\text{\bf SL}}_{2}(\rn)}

\def \SO#1{\text{\bf SO}(#1)}
\def \so#1{\mathfrak{so}(#1)}

\def \U#1{\text{\bf U}(#1)}
\def \u#1{\mathfrak{u}(#1)}

\def \SU#1{\text{\bf SU}(#1)}

\def \Sp#1{\text{\bf Sp}(#1)}
\def \sp#1{\mathfrak{sp}(#1)}

\DeclareMathOperator{\Div}{div} 

\DeclareMathOperator{\trace}{trace}

\numberwithin{equation}{section}
\allowdisplaybreaks

\begin{document}

\title[New iharmonic functions on the Lie groups $\SO n$, $\SU n$, $\Sp n$]
{New Biharmonic functions on the\\ compact Lie groups $\SO n$, $\SU
n$, $\Sp n$}


\author{Sigmundur Gudmundsson}
\address{Mathematics, Faculty of Science\\
University of Lund\\
Box 118, Lund 221\\
Sweden}
\email{Sigmundur.Gudmundsson@math.lu.se}

\author{Anna Siffert}
\address{Max Planck Institute for Mathematics\\
Vivatsgasse 7\\
53111 Bonn\\
Germany}
\email{siffert@mpim-bonn.mpg.de}

\begin{abstract}
We develop a new scheme for the construction of explicit complex-valued proper biharmonic functions on Riemannian Lie groups.  We exploit this and manufacture many infinite series of uncountable families of new solutions on the special unitary group $\SU n$.  We then show that the special orthogonal group $\SO n$ and the quaternionic unitary group $\Sp n$ fall into the scheme. As a by-product we obtain new harmonic morphisms on these groups.  All the constructed maps are defined on open and dense subsets of the corresponding spaces.  
\end{abstract}



\maketitle


\section{Introduction}\label{section-introduction}

In this paper we introduce a new method for constructing infinite families of explicit complex-valued proper biharmonic functions on the Riemannian Lie groups $\SO n$, $\SU n$ and $\Sp n$.
Although the literature on biharmonic functions is vast, the domains of the functions are typically either surfaces or open subsets of flat Euclidean space, with only very few exceptions, see for example \cite{Bai-Far-Oua}. 
The first 
proper biharmonic functions from open subsets of the classical compact simple Lie groups $\SO n$, $\SU n$ and $\Sp n$ have been constructed only recently in \cite{Gud-Mon-Rat-1} by Montaldo, Ratto and the first author.

\smallskip

We first generalise the constructions of \cite{Gud-Mon-Rat-1} and then use the so obtained biharmonic functions as the building blocks for the biharmonic functions constructed in this paper.  Namely, for each such function $f$, we prove that for any positive natural number $d$ there exists a polynomial of the form $$\Phi(f)=\sum_{k=0}^dc_kf^{d-k}\tau(f)^k$$ 
which is proper biharmonic.  Here $\tau(f)$ is the tension field of the function $f$.  We then show that these considerations can be generalised to multi-homogeneous polynomials in $f_1,\dots,f_{\ell}$, where the functions $f_i$ are of the same structure as $f$ above i.e. they are
generalisations of the biharmonic functions constructed in \cite{Gud-Mon-Rat-1}.

\smallskip

Using this construction method we obtain our main result.

\begin{theorem}\label{main-result}
Let $G$ be given by either $\SU n$, $\Sp n$ or $\SO n$ where $n\geq 2$ in the first case, $n\in\mathbb{N}$ in the second case and $n\geq 4$ in the last case.
Then for each choice of $(d_1,\dots,d_{n-1})\in\mathbb{N}^{n-1}$ there exist a proper biharmonic function defined on a dense subset of $G$.
\end{theorem}

As a by-product of our considerations we produce a wealth of new harmonic morphisms from the simple Lie groups $\SO n$, $\SU n$ and $\Sp n$.

\medskip

\noindent{\textbf{Organisation.}}
In Section\,\ref{section-preliminaries} we recall the definitions of biharmonic functions and harmonic morphisms.
The general setting is given in Section\,\ref{section-general-setting}.
In Sections\,\ref{section-SU(n)} to \ref{section-new-U(n)-II} we construct new biharmonic functions on $\SU n$.
We generalise these considerations in Section\,\ref{the-general-case} to  a larger collection of Lie groups.
We show in Sections\,\ref{section-Sp(n)} and \ref{section-SO(n)} 
that $\Sp n$ and $\SO n$ are contained in this larger set of Lie groups and thus construct new proper biharmonic functions on both $\Sp n$ and $\SO n$.
Finally, we provide new harmonic morphisms on $\SO n$, $\SU n$ and $\Sp n$ in Section\,\ref{harmonic-morphisms}.

\medskip

\noindent{\textbf{Acknowledgements.}}
The second author would like to thank the Max Planck Institute for Mathematics in Bonn for support and providing excellent working conditions.


\section{Preliminaries}\label{section-preliminaries}

Let $(M,g)$ be a smooth manifold equipped with a Riemannian metric $g$.
We complexify the tangent bundle $TM$ of $M$ to $T^{\cn}M$ and extend
the metric $g$ to a complex-bilinear form on $T^{\cn}M$.  Then the
gradient $\nabla f$ of a complex-valued function $f:(M,g)\to\cn$ is a
section of $T^{\cn}M$.  In this situation, the well-known linear
{\it Laplace-Beltrami} operator (alt. {\it tension} field) $\tau$ on $(M,g)$
acts on $f$ as follows
$$
\tau(f)=\Div (\nabla f)=\frac{1}{\sqrt{|g|}}
\frac{\partial}{\partial x_j}\left(g^{ij}\, \sqrt{|g|}\,
\frac{\partial f}{\partial x_i}\right).
$$
For two complex-valued functions $f,h:(M,g)\to\cn$ we have the
following well-known relation
\begin{equation}\label{equation-basic}
\tau(fh)=\tau(f)\,h+2\,\kappa(f,h)+f\,\tau(h),
\end{equation}
where the {\it conformality} operator $\kappa$ is given by
$$
\kappa(f,h)=g(\nabla f,\nabla h).
$$
The fact that the operator $k$ is bilinear and the basic property
$$\nabla(f\tilde{f})=\nabla(f)\tilde{f}+f\,\nabla(\tilde{f})$$ of $\nabla$ show that
\begin{equation}\label{kappa-basic}
\begin{aligned}
\kappa(f\tilde{f},h\tilde{h})
&=\tilde{f}\tilde{h}\,\kappa(f,h)+\tilde{f}h\,\kappa(f,\tilde{h})\\
&\qquad +f\tilde{h}\,\kappa(\tilde{f},h)+fh\,\kappa(\tilde{f},\tilde{h}),
\end{aligned}
\end{equation}
where $\tilde{f},\tilde{h}:(M,g)\to\cn$ are complex-valued functions.

\smallskip

For a positive integer $r$, the iterated Laplace-Beltrami operator
$\tau^r$ is defined by
$$
\tau^{0} (f)=f,\quad \tau^r (f)=\tau(\tau^{(r-1)}(f)).
$$
\begin{definition}\label{definition-proper-r-harmonic} For a positive
integer $r$, we say that a complex-valued function $f:(M,g)\to\cn$ is
\begin{enumerate}
\item[(a)] {\it $r$-harmonic} if $\tau^r (f)=0$,
\item[(b)] {\it proper $r$-harmonic} if $\tau^r (f)=0$ and
$\tau^{(r-1)}(f)$ does not vanish identically.
\end{enumerate}
\end{definition}

It should be noted that the {\it harmonic} functions are exactly the
$1$-harmonic and the {\it biharmonic} functions are the $2$-harmonic
ones. By a proper harmonic function we mean a non-constant harmonic function.  In some texts, the $r$-harmonic functions are also called polyharmonic of order $r$.

\begin{remark}
We would like to remind the reader of the fact that a complex-valued function $f:(M,g)\to\cn$ from a Riemannian manifold is a harmonic morphism if it is {\it harmonic} and {\it horizontally conformal} i.e. 
$$\tau(f)=0\ \ \text{and}\ \ \kappa(f,f)=0.$$
The standard reference on this topic is the book \cite{Bai-Woo-book} of Baird and Wood.   We also recommend the regularly updated online bibliography \cite{Gud-bib}.
\end{remark}


\section{The General Setting}\label{section-general-setting}
In this paper we construct biharmonic functions which are rational functions defined on compact Lie groups.
In the first subsection we express $\tau$ and $\kappa$ of rational functions $f=P/Q$ with domains being compact Lie groups in terms of $\tau(P)$, $\tau(Q)$, $\kappa(P,Q)$ and $\kappa(Q,Q)$.
In the second subsection we recall general formulas for $\tau$ and $\kappa$ for functions with domains being compact Lie groups. 
The idea is to use these results to simplify those of the first subsection.
At this point it does not make sense to do so in full generality.
In latter sections we will consider the cases  $\SO n$, $\SU n$ and $\Sp n$ separately, and proceed as indicated above.

\subsection{$\tau$ and $\kappa$ for rational functions on compact Lie groups}

Throughout this paper we work with rational functions of  the complex-valued matrix coefficients of the irreducible standard representations of the compact Lie groups $\SO n$, $\SU n$ and $\Sp n$.

\smallskip

Let $P,Q:G\to\cn$ be two complex-valued functions, $G^*$ be the open and dense subset of $G$ with $$G^*=\{p\in G|\ Q(p)\neq 0\}$$ and $f:G^*\to\cn$ be defined by $f=P/Q$.  
Then a simple calculation, using (\ref{kappa-basic}), shows that the conformality operator $\kappa(f,f)$ satisfies
\begin{eqnarray}\label{kappa-quotient}
Q^4\kappa(f,f)=Q^2\kappa(P,P)-2PQ\kappa(P,Q)+P^2\kappa(Q,Q).
\end{eqnarray}
A similar computation tells us that the tension field $\tau(f)$ fulfils
\begin{eqnarray}\label{tau-quotient}
Q^3\tau(f)=Q^2\tau(P)-2Q\,\kappa(P,Q)+2P\,\kappa(Q,Q)-PQ\,\tau(Q).
\end{eqnarray}

Let $\{e_1,e_2,\dots ,e_n\}$ be a basis for an $n$-dimensional vector space $V$ and
$$\M=\{f_{j\alpha}|\ 1\le j,\alpha\le n\}$$ 
be the set of the matrix coefficients of the action of $G$ on $V$ 
with respect to this basis.  Then it is a consequence of the 
Peter-Weyl theorem that the elements of $\M$ are eigenfunctions 
of the Laplace-Beltrami operator on $G$ all with the same eigenvalue.
Further let $p,q\in\cn^{n\times n}$ and define the two
functions $P,Q:G\to\cn$ by
$$P(g)=\sum_{\alpha,\beta}p_{\alpha\beta}\cdot f_{\alpha\beta}(g)
\ \ \text{and}\ \
Q (g)=\sum_{k,\beta}q_{k\beta}\cdot f_{k\beta}(g).$$
Now that $P,Q:G\to\cn$ are eigenfunctions of the tension field  with the same eigenvalue it follows from equation (\ref{tau-quotient}) that
\begin{equation}\label{tau-quotient-special}
Q^3\tau(f)=2P\,\kappa(Q,Q)-2Q\,\kappa(P,Q).
\end{equation}
\vskip .1cm

\subsection{$\tau$ and $\kappa$ for functions on compact Lie groups}\label{subsection2-general}
Let $G$ be a compact Lie group with Lie algebra $\g$ and $G\to\text{End}(V)$ 
be a faithful irreducible finite dimensional representation of 
$G$.  Then we can identify $G$ with a compact subgroup 
of the general linear group $\GLC n$ where $n$ is the 
dimension of the vector space $V$.  

\medskip

If $Z\in\g$ is a left-invariant vector field on $G$ and $h:U\to\cn$ is a
complex-valued function locally defined on $G$ then the first and
second order derivatives satisfy
\begin{equation*}\label{eq:derivativeZ}
	Z(h)(p)=\frac {d}{ds}[h(p\cdot\exp(sZ))]\big|_{s=0},
\end{equation*}
\begin{equation*}\label{eq:derivativeZZ}
	Z^2(h)(p)=\frac {d^2}{ds^2}[h(p\cdot\exp(sZ))]\big|_{s=0}.
\end{equation*}

The Lie algebra $\glc n$ of the general linear group $\GLC n$ 
can be identified with the set of complex $n\times n$ matrices. 
This carries a natural Euclidean scalar product
$$
g(Z,W)=\Re\trace (Z\cdot W^*)
$$
which induces a left-invariant Riemannian metric $g$ on $\GLC n$.  
Employing the Koszul formula for the Levi-Civita connection $\nabla$
on $(\GLC n,g)$ we see that
\begin{eqnarray*}
	g(\nab ZZ,W)&=&g([W,Z],Z)\\
	&=&\Re\trace (WZ-ZW)Z^*\\
	&=&\Re\trace W(ZZ^*-Z^*Z)^*\\
	&=&g([Z,Z^*],W).
\end{eqnarray*}
Let $[Z,Z^*]_\g$ be the orthogonal projection of the bracket
$[Z,Z^*]$ onto the subalgebra $\g$ of $\glc n$.  Then the above
calculations shows that $$\nab ZZ=[Z,Z^*]_\g.$$ This implies that
the tension field $\tau$ and the conformality operator $\kappa$ are
given by
\begin{equation}\label{tau-kappa-alie-groups}
\tau(h)=\sum_{Z\in\B}Z^2(h)-[Z,Z^*]_\g(f)
\ \ \text{and}\ \
\kappa(h,\tilde{h})=\sum_{Z\in\B}Z(h)Z(\tilde{h}),
\end{equation}
where $\B$ is any orthonormal basis for the Lie algebra $\g$
and $\tilde{h}:U\to\cn$ a
complex-valued function locally defined on $G$.


\section{The Special Unitary Group $\SU n$}\label{section-SU(n)}
This section mainly serves as a preparation for the two sections to follow. First we present some preliminaries and then provide formulae for the tension field $\tau$ and the conformality operator $\kappa$ on $\U n$. Afterwards we construct complex-valued proper biharmonic functions on open and dense subsets of the special orthogonal group $\SU n$.  They are quotients of first order homogeneous polynomials in the matrix coefficients of its standard representation. These results generalise some of those contained in \cite{Gud-Mon-Rat-1}.

\medskip

The unitary group $\U n$ is the compact subgroup of the complex general linear group $\GLC n$ given by
$$
\U n=\{z\in\GLC{n}|\ z\, z^*=I_n\}
$$
with its standard irreducible matrix representation
$$
z=\begin{bmatrix}
z_{11} & \cdots & z_{1n}\\
\vdots & \ddots & \vdots \\
z_{n1} & \cdots & z_{nn}
\end{bmatrix}.
$$
The circle group $\sn^1=\{e^{i\theta}\in\cn | \ \theta\in\rn\}$ acts
on the unitary group $\U n$ by multiplication
$$
(e^{i\theta},z)\mapsto e^{i\theta}\, z
$$
and the orbit space of this action is the special unitary group
$$
\SU n=\{ z\in\U n|\ \det z = 1\}.
$$
The natural projection $\pi:\U n\to\SU n$ is a harmonic morphism with
constant dilation $\lambda\equiv 1$. This has the following interesting consequence.

\begin{proposition}
Let $f:U\to\cn$ be a complex-valued function defined locally on the special unitary group $\SU n$
and $\pi:\U n\to\SU n$ be the natural projection.  Then the composition $f\circ\pi:\pi^{-1}(U)\to\cn$
is a harmonic morphism on $\U n$ if and only if $f:U\to\cn$ is a harmonic morphism on $\SU n$.
\end{proposition}

\begin{proof}
Since the natural projection $\pi:\U n\to \SU n$ is a harmonic morphism the statement follows
directly from Proposition 2.2 of \cite{Gud-Mon-Rat-1} and the fact that the dilation $\lambda$ of $\pi$ satisfies $\lambda\equiv 1$.
\end{proof}

\medskip

Below we use the results of Section\,\ref{subsection2-general} to describe  the tension field $\tau$ and the conformality operator $\kappa$ on the unitary group $\U n$.
This has already been done in \cite{Gud-Sak-1} but we include these considerations for reasons of completeness.

For $1\le r,s\le n$, we shall by $E_{rs}$ denote the real $n\times n$ matrix given by
$$
(E_{rs})_{\alpha\delta}=\delta_{r \alpha}\delta_{s\beta}
$$
and for $r<s$ let $X_{rs},Y_{rs}$ be the symmetric and skew-symmetric matrices
$$
X_{rs}=\frac 1{\sqrt 2}(E_{rs}+E_{sr}),\ \ Y_{rs}=\frac 1{\sqrt 2}(E_{rs}-E_{sr}),
$$ 
respectively.  Further let $D_r$ be the diagonal elements with $$D_r=E_{rr}.$$ 

The standard representation of the Lie algebra $\u n$ of the
unitary group $\U n$ satisfies
$$
\u{n}=\{Z\in\cn^{n\times n}|\ Z+Z^*=0\}
$$
and for this we have the canonical orthonormal basis
$$
\B=\{i X_{rs}, Y_{rs}|\ 1\le r<s\le n\}\cup\{iD_r|\ r=1,\dots ,n\}.
$$
All the elements $Z\in\B$ fulfil the condition $[Z,Z^*]=0$ so it
follows from above that the
Levi-Civita connection satisfies $$\nab ZZ=0.$$ This implies that
for the Laplace-Beltrami operator $\tau$ on the unitary group $\U n$
we have 
$$\tau(f)=\sum_{Z\in\B}Z^2(f).$$

The following result was established in \cite{Gud-Sak-1}.  It describes important properties of the tension field $\tau$ and the conformality operator $\kappa$ on the unitary group $\U n$.  

\begin{lemma}\label{lemma-U(n)}
For $1\le j,\alpha\le n$, let $z_{j\alpha}:\U n\to\cn$ be the complex-valued matrix
coefficients of the standard representation of $\U n$. Then the following relations hold
\begin{equation*}
\tau(z_{j\alpha})= -n\cdot z_{j\alpha}\ \ \text{and}\ \
\kappa(z_{j\alpha},z_{k\beta})= -z_{k\alpha}z_{j\beta}.\\
\end{equation*}
\end{lemma}

\smallskip

The next stated result is a direct consequence of Lemma \ref{lemma-U(n)}.

\begin{lemma}\label{lemma-columns-dependent}
Let $M_Q$ be the following non-zero complex matrix
\begin{equation*}
M_Q=
\begin{bmatrix}
q_{11} & \cdots & q_{1n} \\
\vdots & \ddots & \vdots \\
q_{n1} & \cdots & q_{nn}
\end{bmatrix}\ \
\end{equation*}
and $Q:\U n\to\cn$ be the polynomial function on the unitary group given by
$$
Q(z)=\sum_{j,\alpha}q_{j\alpha}z_{j\alpha}.
$$
Then the equation $Q^2+\kappa(Q,Q)=0$ is fulfilled if and only if the columns of $M_Q$ are pairwise linearly dependent.
\end{lemma}

\begin{proof}
The statement follows easily from
$$
Q^2+\kappa(Q,Q)=\sum_{j,k,\alpha,\beta}
(q_{j\alpha}q_{k\beta}-q_{k\alpha}q_{j\beta})	z_{j\alpha}z_{k\beta}=0.
$$
\end{proof}

Note that in this paper we will always assume that the equation $Q^2+\kappa(Q,Q)=0$ is fulfilled.
This is done for reasons of simplicity. 
The purpose of the following discussion is to explain how the polynomial functions $P,Q:\U n\to\cn$ are chosen in different situations in the remainder of this paper.

\smallskip 

Let us assume that the function $Q:\U n\to\cn$ in Lemma  \ref{lemma-columns-dependent} satisfies 
$$Q^2+\kappa(Q,Q)=0$$ and that the $\beta$-th column vector of $M_Q$ is non-zero.  Then there exists a non-zero vector $a=(a_1,a_2,\dots ,a_n)\in\cn^n$ such that $M_Q$ is of the form
\begin{equation*}
M_Q=
\begin{bmatrix}
a_1q_{1\beta} & a_2q_{1\beta} & \cdots & a_nq_{1\beta} \\
a_1q_{2\beta} & a_2q_{2\beta} & \cdots & a_nq_{2\beta} \\
\vdots        & \vdots        & \vdots & \vdots        \\	
a_1q_{n\beta} & a_2q_{n\beta} & \cdots & a_nq_{n\beta}
\end{bmatrix}.
\end{equation*}
Hence there exists a non-zero complex vector, namely 
$$
q=(q_1,q_2,\dots ,q_n)=(q_{1\beta},\dots,q_{n\beta}),
$$ 
such that the function $Q$ is of the form 
$$
Q(z)=\sum_{j,\alpha} q_{j}a_\alpha z_{j\alpha}.
$$

The next theorem shows how the standard representation of $\U n$ can
be used to produce proper biharmonic functions on the special unitary group $\SU n$. This result generalises Theorem 4.2 in \cite{Gud-Mon-Rat-1}.  Although its proof can be obtained by easy modifications of the original one we provide here another one for the reader's convenience.

\begin{theorem}\label{theorem-SU(n)}
Let $a,q\in\cn^n$ be two non-zero vectors and $M_P$ be the following non-zero complex matrix
\begin{equation*}
M_P=
\begin{bmatrix}
p_{11} & \cdots & p_{1n}\\
\vdots & \ddots & \vdots\\
p_{n1} & \cdots & p_{nn}
\end{bmatrix}.
\end{equation*}
Further let the polynomial functions $P,Q:\U n\to\cn$ be given by
$$P(z)=\sum_{j,\alpha}p_{j\alpha}z_{j\alpha}\ \ \text{and}\ \ Q(z)=\sum_{k,\beta}q_{k}a_\beta z_{k\beta}$$
and the rational function $f=P/Q$ be defined on the open and dense subset $\{z\in\U n|\ Q(z)\neq 0\}$ of $\U n$.  Then we have the following.
\begin{enumerate}
\item[(1)] The function $f$ is harmonic if and only if $PQ+\kappa(P,Q)=0$.  This is equivalent to (i) the vector $q$ and each column vector of the matrix $M_P$ are linearly dependent or (ii) the vector $a$ and the matrix $M_P$ are of the following special form
$$a=[0,\dots,0,a_{\beta_0},0,\dots,0],$$
\begin{equation*}
M_P=
\begin{bmatrix}
     0 & \cdots & 0      & p_{1\beta_0}  & 0      & \cdots & 0 \\
\vdots & \vdots & \vdots & \vdots & \vdots & \vdots & 0 \\
     0 & \cdots & 0      & p_{n\beta_0}  & 0      & \cdots & 0 \\
\end{bmatrix}.
\end{equation*}
\item[(2)] The function $f$ is proper biharmonic if and only if
$PQ+\kappa(P,Q)\neq 0$ i.e. if and only if neither (i) nor (ii) of (1) is satisfied.  
\end{enumerate}
The corresponding statements hold for the function induced on the special unitary group $\SU n$.
\end{theorem}

\begin{proof}
It is an immediate consequence of the equations (\ref{tau-quotient-special}) and $$\kappa(Q,Q)=-Q^2$$ that the tension field $\tau(f)$ satisfies 
\begin{eqnarray}\label{tau-f}
\tau(f)=-2(PQ+2\kappa(P,Q))Q^{-2}.
\end{eqnarray}
This means that the function $f$ is harmonic if and only if $PQ+\kappa(P,Q)=0$, or equivalently, 
\begin{equation}\label{tau-f-coord}
\sum_{j,k,\alpha,\beta}(p_{j\alpha}q_k-p_{k\alpha}q_j)
a_{\beta}z_{j\alpha}z_{k\beta}=0.
\end{equation}

Let us first investigate the special case when the vector $a$ and the matrix $M_P$ are of the following special form
$$a=[0,\dots,0,a_{\beta_0},0,\dots,0]$$ and 
\begin{equation*}
M_P=
\begin{bmatrix}
     0 & \cdots & 0      & p_{1\beta_0}  & 0      & \cdots & 0 \\
\vdots & \vdots & \vdots & \vdots & \vdots & \vdots & 0 \\
     0 & \cdots & 0      & p_{n\beta_0}  & 0      & \cdots & 0 \\
\end{bmatrix}.
\end{equation*}
Then the equation (\ref{tau-f-coord}) reduces to the following which is trivially satisfied
\begin{equation*}
\sum_{j,k}(p_{j\beta_0}\,q_k-p_{k\beta_0}\,q_j)
a_{\beta_0}z_{j\beta_0}z_{k\beta_0}=0.
\end{equation*}
If we are not is the special situation, just discussed, then 
$$(p_{j\alpha}q_k-p_{k\alpha}q_j)a_{\beta}=0$$
for all $j,k,\alpha,\beta$.  Since $a\neq 0$ there exists an $a_\beta\neq 0$ and hence
$$(p_{j\alpha}q_k-p_{k\alpha}q_j)=0$$
for all $j,k,\alpha$.
This shows that the non-zero vector $q$ and any column of the matrix $M_P$ are linearly dependent.  We have now proven the statement (1).

\smallskip

At this point it is convenient to introduce the following polynomial functions $R_\alpha,S_\alpha:\U n\to\cn$ satisfying
\begin{equation*}
	R_{\alpha}(z)=\sum_{j,\beta=1}^n p_{j\alpha}a_{\beta}z_{j\beta}
	\ \ \text{and}\ \ 
	S_{\alpha}(z)=\sum_{k=1}^nq_kz_{k\alpha}.
\end{equation*}	
Then it is a direct consequence of Lemma \ref{lemma-U(n)} that
\begin{equation*}
\kappa(P,Q)=-\sum_{\alpha=1}^nR_{\alpha}S_{\alpha}.
\end{equation*}
By substituting this into equation (\ref{tau-f}) we obtain
\begin{equation*}
\tau(f)=2(\sum_{\alpha=1}^nR_{\alpha}S_{\alpha}-PQ)Q^{-2}.
\end{equation*}
Exploiting equations (\ref{equation-basic}), (\ref{kappa-basic}) and the fact that $\tau(Q^{-2})=2(n-3)Q^{-2}$ we now yield 
\begin{multline*}
\tau^2(f)=2\sum_{\alpha=1}^n(\tau(R_{\alpha})S_{\alpha}+2\kappa(R_{\alpha},S_{\alpha})+R_{\alpha}\tau(S_{\alpha}))Q^{-2}\\-2(\tau(P)Q+2\kappa(P,Q)+P\tau(Q))Q^{-2}\\-8\kappa(\sum_{\alpha=1}^nR_{\alpha}S_{\alpha}-PQ,Q)Q^{-3}+4(n-3)(\sum_{\alpha=1}^nR_{\alpha}S_{\alpha}-PQ)Q^{-2}.
\end{multline*}
Again employing Lemma \ref{lemma-U(n)} we easily see that 
\begin{equation*}
\kappa(\sum_{\alpha=1}^nR_{\alpha}S_{\alpha}-PQ,Q)=-Q\sum_{\alpha=1}^nR_{\alpha}S_{\alpha}-PQ^2.
\end{equation*}
Using this and the fact that $P,Q,R_{\alpha}$ and $S_{\alpha}$ are eigenfunctions of $\tau$ with eigenvalue $\lambda=-n$ we then have
\begin{multline*}
\tau^2(f)=(-4n\sum_{\alpha=1}^nR_{\alpha}S_{\alpha}-4PQ+4nPQ+4\sum_{\alpha=1}^nR_{\alpha}S_{\alpha})Q^{-2}\\
+4(n-3)(\sum_{\alpha=1}^nR_{\alpha}S_{\alpha}-PQ)Q^{-2}+8(\sum_{\alpha=1}^nR_{\alpha}S_{\alpha}-PQ)Q^{-2}=0.
\end{multline*}
This establishes the statement claimed in (2).
\end{proof}


\section{Biharmonic Functions on $\SU n$ - (I)}\label{section-new-U(n)-I}

In this section we manufacture an infinite sequence of new proper biharmonic functions defined on open and dense subsets of the special unitary group $\SU n$.

\medskip

Our strategy is the following: Let $a,b,p,q\in\cn^n$ be non-zero and the polynomial functions $P,Q:\U n\to\cn$ be given by
\begin{equation*}
P(z)=\sum_{j}p_{j}a_\alpha z_{j\alpha}\ \ \text{and}\ \  Q(z)=\sum_{k}q_{k}b_\beta z_{k\beta},
\end{equation*}
such that $PQ+\kappa (P,Q)\neq 0$.  Then it is clear from Theorem \ref{theorem-SU(n)} that for $c_0,c_1\in\cn$ with $c_0\neq 0$ 
the function 
$$
\Phi_1(f,\tau(f))=c_0f+c_1\tau(f)
$$
induces a proper biharmonic function locally defined on the special unitary group $\SU n$.   The function $\Phi_1$ is a homogeneous first order polynomial in $f$ and its tension field $\tau(f)$.  Our aim is now to generalise this to any positive degree.

\smallskip

Let $d$ be a positive integer and $\Phi_d:\cn^2\to\cn$ be a complex homogeneous polynomial
of the form
$$\Phi_d(z_1,z_2)=\sum_{k=0}^dc_k\,z_1^{d-k}z_2^k$$
We are now interested in determining all such polynomials with the property that the
function $\Phi_d(f,\tau(f))$ is proper biharmonic i.e.
$$\tau(\Phi_d(f,\tau(f)))\neq 0\ \ \text{and}\ \ \tau^2(\Phi_d(f,\tau(f)))=0.$$
Before we can do this we need some practical preparations.

\begin{lemma}\label{basic-un}
Let $a,b,p,q\in\cn^n$ be non-zero and the polynomial functions $P,Q:\U n\to\cn$ be given by
\begin{equation*}
P(z) =\sum_{j\alpha}p_{j}a_\alpha z_{j\alpha}
\ \ \text{and}\ \  
Q (z)=\sum_{k\beta}q_{k}b_\beta z_{k\beta}.
\end{equation*}
Then their rational quotient $f=P/Q$ satisfies
\begin{enumerate}
\item[(1)]$\kappa(f,f)=f\,\tau(f)$,
\item[(2)] $\kappa(f,\tau(f))=\tau(f)^2$,
\item[(3)] $\kappa(\tau(f),\tau(f))=-2\tau(f)^2$.
\end{enumerate}
\end{lemma}

\begin{proof}
Here it is convenient to introduce the complex-valued polynomial functions $R,S:\U n\to\cn$ with
\begin{equation*}
R(z)=\sum_{j,\beta}p_{j}b_\beta z_{j\beta}\ \ \text{and}\ \  S(z)=\sum_{k,\alpha}q_{k}a_\alpha z_{k\alpha}.
\end{equation*}
Then an elementary computation, applying Lemma \ref{lemma-U(n)}, yields
$$
\kappa(P,P)=-P^2,\ \ 
\kappa(Q,Q)=-Q^2,
$$
$$
\kappa(R,R)=-R^2,\ \ 
\kappa(S,S)=-S^2,
$$
$$
\kappa(P,Q)=-RS,\ \ 
\kappa(R,S)=-PQ,  
$$
$$
\kappa(P,R)=-PR,\ \ 
\kappa(P,S)=-PS,
$$
$$
\kappa(Q,R)=-QR,\ \ 
\kappa(Q,S)=- QS.
$$
With this at hand a straightforward calculation establishes the claim.
\end{proof}

As a consequence of Lemma \ref{basic-un} we now have the following useful result.

\begin{lemma}\label{lemma-kappa-tau-U(n)}
Let $a,b,p,q\in\cn^n$ be non-zero and the polynomial functions $P,Q:\U n\to\cn$ be given by
\begin{equation*}
P(z) =\sum_{j,\alpha}p_{j}a_\alpha z_{j\alpha}\ \ \text{and}\ \  Q (z)=\sum_{k,\beta}q_{k}b_\beta z_{k\beta}.
\end{equation*}
If $\ell,m$ are positive integers then the rational function $f=P/Q$ satisfies
\begin{enumerate}
\item[(1)] 
$\kappa (f^\ell,\tau (f)^m)=\ell\, m f^{\ell-1}\tau(f)^{m+1}$,
\item[(2)] 
$\kappa (f^\ell,f^m)=\ell\, m\, f^{m+\ell-1}\tau(f)$,
\item[(3)] 
$\tau (f^\ell) = \ell^2 f^{\ell-1}\tau(f)$,
\item[(4)] 
$\kappa(\tau(f)^\ell,\tau(f)^m)=-2\,\ell\,m\,\tau (f)^{\ell+m}$,
\item[(5)] 
$\tau(\tau (f)^ \ell)=-2\,\ell(\ell-1)\,\tau (f)^\ell$,
\end{enumerate}
\end{lemma}

\begin{proof}
Let $\B$ be the standard orthonormal frame for the tangent bundle of $\U n$.  Then statement (1) is an immediate  consequence of the following computation
\begin{eqnarray*}
\kappa(f^\ell,\tau(f)^m)=&=&\sum_{X\in\B}X(f^\ell)X(\tau(f)^m)\\
&=&\sum_{X\in\B}\ell f^{\ell-1}X(f)\, m\,\tau(f)^{m-1}X(\tau(f))\\
&=&\ell\,mf^{\ell-1}\tau(f)^{m-1}\kappa(f,\tau(f))\\
&=&\ell\,mf^{\ell-1}\tau(f)^{m+1}.
\end{eqnarray*}
The proof of (2) follows similarly and is therefore skipped. It is clear that (3) is true for $\ell=1$ and the statement is a direct consequence of the following induction step
\begin{eqnarray*}
\tau(f^{\ell+1})&=&\tau(f)f^\ell+2\kappa(f,f^\ell)+f\tau(f^\ell)\\
&=&f^\ell\tau(f)+2\ell f^{\ell-1}f\tau(f)+f\ell^2f^{\ell-1}\tau(f)\\
&=&(\ell+1)^2f^\ell\tau(f).
\end{eqnarray*}
The equation (4) follows immediately from
\begin{eqnarray*}
\kappa(\tau(f)^\ell,\tau(f)^m)&=&\ell\,\tau(f)^{\ell-1}\ m\,\tau(f)^{m-1}\kappa(\tau(f),\tau(f))\\
&=&-2\ell\,m\,\tau(f)^{\ell+m-2}\tau(f)^2\\
&=&-2\ell\, m\,\tau(f)^{\ell+m}.
\end{eqnarray*}
The statement (5) is clearly true when $\ell=1$ and the rest is a consequence of the next induction step
\begin{eqnarray*}
\tau(\tau(f)^{\ell+1})&=&\tau(\tau(f))\tau(f)^\ell+2\kappa(\tau(f),\tau(f)^\ell))+\tau(f)\tau(\tau((f)^\ell))\\
&=&2\ell\tau(f)^{\ell-1}\kappa(\tau(f),\tau(f))-2\ell(\ell-1)\tau(f)^{\ell+1}\\
&=&-2\ell(2\tau(f)^{\ell+1}+(\ell-1)\tau(f)^{\ell+1})\\
&=&-2(\ell+1)\ell\tau(f)^{\ell+1}.
\end{eqnarray*}
\end{proof}

After our preparations we are now ready to construct the new proper biharmonic functions promised at the beginning of this section. Before attacking the general case we first consider explicit examples.  For the remainder of this section let the function $f$ be given as in Lemma \ref{lemma-kappa-tau-U(n)}.

\begin{example}
We first investigate the case of second order homogeneous polynomials
$$\Phi_2(f)=c_0\,f^2 + c_1\, f\,\tau(f) + c_2\, \tau(f)^2.$$
Then a simple application of Lemma \ref{basic-un} shows that
$$\tau(\Phi_2(f))=4\,c_0f\,\tau(f)+(3\,c_1 -4\,c_2)\,\tau(f)^2,$$
so $\Phi_2$ is a harmonic function if and only if 
$
c_0=0\ \ \text{and}\ \ 4c_2=3c_1.
$ 
Hence the function $\Phi_2$ is proper harmonic if and only if it is a non-zero multiple of  
$$
H_2(f)=4f\tau(f)+3\tau(f)^2.
$$  
If we now apply the Laplace-Betrami operator again we obtain
$$\tau^2(\Phi_2(f))=4\,(3\,c_0 - 3\,c_1 + 4\,c_2)\,\tau(f)^2.$$
This tells us that the function $\Phi_2$ is proper biharmonic if and only if
$$4\,c_2 = 3\,c_1-3\,c_0\ \ \text{and}\ \ c_0\neq 0.$$
This statement is clearly equivalent to: $\Phi_2$ is proper biharmonic if and only if $c_0\neq 0$ and $\Phi_2=c_0B_2+c_1H_2$, where 
$$
B_2(f)=4\,f^2-3\,\tau(f)^2\ \ \text{and}\ \ H_2(f)=4\,f\tau(f)+3\,\tau(f)^2.
$$

The same method can now be applied to show that for $d=3,4$ every proper biharmonic function $\Phi_d(f)$ of the form 
$$
\Phi_d(f)=\sum_{k=0}^dc_k\,f^{d-k}\tau(f)^k
$$
is given by $\Phi_d=c_0\,B_d+c_1\,H_d$, where 
$$
B_3(f)=(6\,f^3-27\,f\tau(f)^2-15\,\tau(f)^3),
$$
$$
H_3(f)=(6\,f^2\tau(f) + 12\,f\tau(f)^2+5\,\tau(f)^3)
$$
and
$$
B_4(f)=(32\,f^4-480\,f^2\tau(f)^2-640\,f\tau(f)^3-210\,\tau(f)^4),
$$
$$
H_4(f)=(32\,f^3\tau(f)+120\,f^2\tau(f)^2+120\,f\tau(f)^3+35\,\tau(f)^4).
$$
\end{example}

\vskip .4cm

After studying the cases when $d=2,3,4$ we now consider the general situation.  As a first intermediate step we investigate the harmonic functions. 

\begin{proposition}\label{prop-harmonic}
Let $a,b,p,q\in\cn^n$ be non-zero, the polynomial functions $P,Q:\U n\to\cn$ be given by
\begin{equation*}
	P(z) =\sum_{j\alpha}p_{j}a_\alpha z_{j\alpha}
	\ \ \text{and}\ \  
	Q (z)=\sum_{k\beta}q_{k}b_\beta z_{k\beta}
\end{equation*}	
and $f=P/Q$ be their rational quotient. Then the function 
\begin{equation*}
\Phi_d(f)=\sum_{k=0}^dc_kf^{d-k}\tau(f)^k
\end{equation*}
is proper harmonic if and only if $c_0=0$, $c_1\neq 0$ and for $k=1,\dots, d-1$
$$2\,k\,(k+1)\,c_{k+1}=(d^2-k^2)c_k.$$
\end{proposition}

\begin{proof}
An elementary computation, applying Lemma \ref{lemma-kappa-tau-U(n)}, yields  
\begin{equation}\label{equa-tau-Phi}
\tau(\Phi_d(f))=\sum_{k=0}^dc_k
\Bigl[(d^2-k^2)\tau(f)
-2k(k-1)f\Bigr]f^{d-k-1}\tau(f)^{k}.
\end{equation}
The condition $\tau(\Phi_d(f))=0$ is clearly equivalent to $c_0=0$ and the first order linear difference equation 
$$2\,k\,(k+1)\,c_{k+1}=(d^2-k^2)c_k,$$
for $k=1,\dots, d-1.$  The statement is a direct consequence of these relations.
\end{proof}

Let us now assume that $d$ is a positive integer and that  $\Phi_d(f)$ is a proper biharmonic function of the form
\begin{equation*}
\Phi_d(f)=\sum_{k=0}^dc_kf^{d-k}\tau(f)^k.
\end{equation*}
It then follows from the identity (\ref{equa-tau-Phi}) and Lemma
\ref{lemma-kappa-tau-U(n)} that 
\begin{eqnarray*}
\tau^2(\Phi_d(f))&=&\sum_{k=0}^dc_k
\Bigl[(d^2-k^2)(d^2-(k+1)^2)f^{d-k-2}\tau(f)^{k+2}\\
& &\qquad\qquad -4k^2(d^2-k^2)f^{d-k-1}\tau(f)^{k+1}\\
& &\qquad\qquad\qquad\qquad +4k^2(k-1)^2f^{d-k}\tau(f)^k\Bigr].
\end{eqnarray*}
By comparing the coefficients of $\tau^2(\Phi_d(f))=0$ we obtain the following second order linear difference equation
\begin{eqnarray*}
4(k-1)^2k^2c_k&=&4(k-1)^2(d^2-(k-1)^2)\,c_{k-1}\\
& &\qquad -\,(d^2-(k-2)^2)(d^2-(k-1)^2)\,c_{k-2},
\end{eqnarray*}
for $2\leq k\leq d$. Together with Proposition \ref{prop-harmonic} this shows that $c_0\neq 0$ and that  $c_2,\dots,c_d$ are determined by $c=(c_0,c_1)\in\cn^2$.

\medskip

Let $B_d(f)$ and $H_d(f)$ be the functions obtained this way with $c=(1,0)$ and $(1,0)$, respectively. Then $B_d(f)$ is proper biharmonic and $H_d(f)$ is proper harmonic.  We have shown that every proper biharmonic function of the form 
\begin{equation*}
\Phi_d(f)=\sum_{k=0}^dc_kf^{d-k}\tau(f)^k
\end{equation*}
can be written as a linear combination 
$$\Phi_d(f)=c_0\,B_d(f)+c_1\,H_d(f),$$ where $c_0,c_1\in\cn$ such that $c_0\neq 0$.  
Thus we have established the following result. This is a special case of Theorem\,\ref{main-result}. 

\begin{theorem}\label{main-d1}
For each $d\in\mathbb{N}^+$ there exist a proper biharmonic function of the form
$$\Phi_d(f,\tau(f))=\sum_{k=0}^dc_k\,f^{d-k}\tau(f)^k$$
 defined on an open and dense subset of $\SU n$.
\end{theorem}


\section{Biharmonic Functions on  $\SU n$ - (II)}
\label{section-new-U(n)-II}

In this section we continue our constructions of local biharmonic functions on the special unitary group $\SU n$.  
Namely, we will construct biharmonic multi-homogeneous polynomials
of the form
\begin{equation*}
\Phi_{d_1,\dots,d_m}(f_1,\dots,f_m)=\sum_{k_1=0}^{d_1}\dots\sum_{k_m=0}^{d_m}c_{k_1,\dots,k_m}f_1^{d_1-k_1}\tau(f_1)^{k_1}\dots f_m^{d_m-k_m}\tau(f_m)^{k_m},
\end{equation*}
where the functions $f_i$ are carefully chosen rational functions.
In the present section we only deal with two examples.
Namely, we will construct two biharmonic two-homogeneous polynomials.
These considerations will help the reader to understand the calculations of Section\,\ref{the-general-case}
in which we deal with the general case.

\medskip

Let $p,q\in\cn^n$ be linearly independent and define the complex-valued polynomial functions $P_\alpha,Q_\beta:\U n\to\cn$ by 
$$
P_\alpha=\sum_jp_j a_{\alpha}z_{j\alpha}\ \ \text{and}\ \  Q_\beta=\sum_kq_k b_{\beta}z_{k\beta}.
$$
For a fixed $\beta$, let $W_\beta$ be the open and dense subset $\{z\in\U n|\ Q_\beta (z)\neq 0\}$ of $\U n$ and define the functions $f_1,\dots,f_{n}:W_\beta\to\cn$ by 
$$f_1=P_1/Q_\beta,\dots,f_{n}=P_{n}/Q_\beta.$$  
Note that according to Theorem \ref {theorem-SU(n)} the function $f_{i}:\U n\to\cn$ is harmonic if $i=\beta$ and proper biharmonic otherwise. 

\smallskip

The following lemma generalises Lemma\,\ref{basic-un}.

\begin{lemma}\label{kappa-alpha-beta}
In the above situation, the tension field $\tau$ and the conformality operator $\kappa$ satisfy the following identities for any $i,j\in\{1,\dots,n\}$
\begin{enumerate}
\item[(1)] 
$2\kappa(f_{i},f_{j})
=f_{j}\tau(f_{i})+\tau(f_{j})f_{i}$,
\item[(2)] 
$\kappa(f_{i},\tau(f_{j}))
=\tau(f_{i})\tau(f_{j})$,
\item[(3)] 
$\kappa(\tau(f_{i}),\tau(f_{j}))
=-2\tau(f_{i})\tau(f_{j})$,
\item[(4)] 
$\tau(f_i^m)=m^2f_i^{m-1}\tau(f_i)$,
\item[(5)] 
$\tau(\tau(f_i)^m)=-2\,m\,(m-1)\,\tau(f_i)^m$.
\end{enumerate}
\end{lemma}

Then a repeated application of Lemma \ref{kappa-alpha-beta} provides the next result.

\begin{lemma}\label{cor}
For any $i,j\in\{1,\dots,n\}$ the conformality operator $\kappa$ satisfies the following identities
\begin{enumerate}
\item[(1)] $2\kappa(f_i^\ell,f_j^m)=\ell\,m\,f_i^{\ell-1}f_j^{m-1}(f_i\tau(f_j)+\tau(f_i)f_j)$,
\item[(2)] $\kappa(f_i^\ell,\tau(f_j)^m)=\ell\,m\,f_i^{\ell-1}\tau(f_i)\tau(f_j)^m$,
\item[(3)] $\kappa(\tau(f_i)^\ell,\tau(f_j)^m)=-2\,\ell\,m\,\tau(f_i)^\ell\tau(f_j)^m$.
\end{enumerate}
\end{lemma}

With these preparations at hand we can now construct proper biharmonic functions.
Below we will deal with two examples.

\begin{example}\label{example-V}
Let $V$ be the $4$-dimensional complex vector space with basis 
$$\B=
\{
f_1f_2,\,
f_1\tau(f_2),\,
\tau(f_1)f_2,\,
\tau(f_1)\tau(f_2)
\}.
$$
Then the restriction $T$ of the Laplace-Beltrami operator $\tau$ to $V$ is a linear endomorphism $T:V\to V$ of $V$ and its kernel consists of the harmonic functions in $V$.  Let $M_T$ be the matrix of $T$ with respect to the basis $\B$.  Then a simple calculation shows that for each $c=(c_1,c_2,c_3,c_4)\in\cn^4$ we have 
\begin{equation*}
M_T\cdot c=
\begin{bmatrix}
0 & 0 & 0 &  0 \\
2 & 0 & 0 &  0 \\
2 & 0 & 0 &  0 \\
0 & 3 & 3 & -4 \\
\end{bmatrix}
\cdot 
\begin{bmatrix}
c_1\\
c_2\\
c_3\\
c_4\\
\end{bmatrix}
=
\begin{bmatrix}
0\\
2c_1\\
2c_1\\
3c_2+3c_3-4c_4
\end{bmatrix}
\end{equation*}
This means that every harmonic function $H(f_1,f_2)$ in $V$ is of the form 
$$H(f_1,f_2)=4\,(c_2\,f_1\tau(f_2)+c_3\,\tau(f_1)f_2)+3\,(c_2+c_3)\,\tau(f_1)\tau(f_2).$$
We have therefore constructed a complex two dimensional family of local harmonic functions on the special unitary group $\SU n$.

The biharmonic elements of $V$ form the kernel of $T^2$.  They can be determined by solving the following linear system
\begin{equation*}
M^2_T\cdot c=
\begin{bmatrix}
0  & 0   &   0 &  0 \\
0  & 0   &   0 &  0 \\
0  & 0   &   0 &  0 \\
12 & -12 & -12 & 16 \\
\end{bmatrix}
\cdot
\begin{bmatrix}
c_1\\
c_2\\
c_3\\
c_4\\
\end{bmatrix}
=4\cdot
\begin{bmatrix}
0\\
0\\
0\\
3c_1-3c_2-3c_3+4c_4
\end{bmatrix}.
\end{equation*}
From this we yield a complex three dimensional family of local biharmonic functions on the special unitary group $\SU n$.  Each such function is of the form
$$B(f_1,f_2)=4\,(c_1f_1f_2+c_2f_1\tau(f_2)
+c_3\tau(f_1)f_2)+3(c_2+c_3-c_1)\tau(f_1)\tau(f_2).$$
The reader should note that the function $B(f_1,f_2)$ is proper biharmonic if and only if $c_1\neq 0$.
\end{example}

\begin{example}
Let us now consider the six dimensional complex vector space $W$ with basis 
$$\B=
\{
f_1^2f_2,\,
f_1^2\tau(f_2),\,
f_1\tau(f_1)f_2,\,
f_1\tau(f_1)\tau(f_2),\,
\tau(f_1)^2f_2,\ 
\tau(f_1)^2\tau(f_2)
\}.
$$
We can now employ the same method as in Example \ref{example-V} and find that the harmonic functions in $W$ form a two dimensional subspace and are of the form
\begin{eqnarray*}
H(f_1,f_2)&=&c_1f_1^2f_2+c_2f_1^2\tau(f_2)
+c_3f_1\tau(f_1)f_2+c_4f_1\tau(f_1)\tau(f_2)\\
& &\qquad +c_5\tau(f_1)^2f_2+c_6\tau(f_1)^2\tau(f_2),
\end{eqnarray*}
where $c_1=0$, $c_4=2\,c_2+c_3$, $c_5=c_3$ and $6\,c_6=5\,c_2+5\,c_3$.
	
\medskip
	
By studying the bitension field $\tau^2$ it is not difficult to see that the biharmonic functions in $W$ form a complex three dimensional family.  They are of the form 
\begin{eqnarray*}
B(f_1,f_2)&=&c_1f_1^2f_2+c_2f_1^2\tau(f_2)
+c_3f_1\tau(f_1)f_2+c_4f_1\tau(f_1)\tau(f_2)\\
& &\qquad +c_5\tau(f_1)^2f_2+c_6\tau(f_1)^2\tau(f_2),
\end{eqnarray*} 
where the coefficients satisfy the following linear conditions
\begin{eqnarray*}
c_4&=&2\,c_2+c_3-3\,c_1,\\
2\,c_5&=&2\,c_3-3\,c_1,\\
6\,c_6&=&5\,c_2+5\,c_3-15\,c_1.
\end{eqnarray*}
As in Example \ref{example-V}, the function $B(f_1,f_2)$ is proper biharmonic if and only if $c_1\neq 0$.
\end{example}

As already mentioned above we will now not go on with generalising Theorem\,\ref{main-d1} to the multi-homogeneous polynomial setting. We will postpone this to the next section in which we 
deal with the construction of biharmonic functions not just on $\SU n$ but on a larger collection of Lie groups.


\section{Biharmonic Functions on compact Lie groups}\label{the-general-case}
In this section we generalise the considerations of Sections \ref{section-new-U(n)-I} and \ref{section-new-U(n)-II} to Lie groups  for which there exist eigenfunctions of the Laplace-Beltrami operator which satisfy several additional conditions.
These conditions are chosen such that an analogue of Lemma\,\ref{cor} holds.

\medskip

Let $G$ be a compact Lie subgroup of $\GLC n$ and for $N\in\mathbb{N}$ and $1\leq j\leq N$, let $P_j,Q,R,S_j:G\to\cn$ be eigenfunctions of the Laplace-Beltrami operator $\tau$ all with the same eigenvalue $\lambda$. Moreover, let $\mu$ be a constant such that the conformality operator $\kappa$ satisfies 
\begin{equation}\label{PP-QQ}
\begin{aligned}
&
\kappa(P_j,P_k)=\mu\, P_jP_k,\ \ 
\kappa(S_j,S_k)=\mu\, S_jS_k,\\
&
\kappa(Q,Q)=\mu\, Q^2,\ \ 
\kappa(R,R)=\mu\, R^2,\\ 
&
\kappa(Q,R)=\mu\, QR,\ \ 
\kappa(Q,S_j)=\mu\, QS_j,\\
&
\kappa(P_j,R)=\mu\, P_jR,\ \
\kappa(P_j,S_k)=\mu\, P_kS_j,\\
&
\kappa(P_j,Q)=\mu\, RS_j,\ \ 
\kappa(R,S_j)=\mu\, P_jQ.
\end{aligned}
\end{equation}
Further let $f_j:G^*\to\cn$ be the quotient $f_j=P_j/Q$ defined on the open and dense subset $G^*=\{p\in G|\ Q(p)\neq 0\}$ of $G$. 

\begin{remark}
At the first glance the conditions (\ref{PP-QQ}) might seem rather restrictive.
However, we will see at the end of this and in the following two sections that we can easily construct plenty of such functions $P_j,Q,R,$ and $S_j$ on $\SU n, \Sp n$ and $\SO n$, respectively, which satisfy these conditions. 
\end{remark}

\smallskip

Using conditions (\ref{PP-QQ}), a tedious but straightforward computation similar to those in the proof of Lemma\,\ref{lemma-kappa-tau-U(n)} yields the next result.

\begin{lemma}\label{basic-identities}
If $m,\ell$ are positive integers and $j,k\in\{1,\dots,N\}$ then the conformality operator $\kappa$ satisfies the following identities
\begin{enumerate}
\item[(1)] $2\kappa(f_i^m,f_j^\ell)=m\,\ell\,f_i^{m-1}f_j^{\ell-1}(f_i\tau(f_j)+\tau(f_i)f_j)$,
\item[(2)] $\kappa(f_i^m,\tau(f_j)^\ell)=m\,\ell\,f_i^{m-1}\tau(f_i)\tau(f_j)^\ell$,
\item[(3)] $\kappa(\tau(f_i)^m,\tau(f_j)^\ell)=2\,\mu\,m\,\ell\,\tau(f_i)^m\tau(f_j)^\ell$,
\item[(4)] 
$\tau(f_i^m)=m^2f_i^{m-1}\tau(f_i)$,
\item[(5)] 
$\tau(\tau(f_i)^m)=2\,\mu\,m\,(m-1)\,\tau(f_i)^m$.
\end{enumerate}
\end{lemma}

We have now gathered all the tools we need for the construction of biharmonic multi-homogeneous polynomials on $G$.
As an intermediate step we will however first construct a wealth of \textit{harmonic} multi-homogeneous polynomials on $G$.

\begin{theorem}\label{general-harmonic}
Let $1\leq m\leq N$ be given and $f_i=P_i/Q$, $i=1,\dots, m$ be proper biharmonic functions. Then the function 
$$F=\sum_{k_1=0}^{d_1}\dots\sum_{k_m=0}^{d_m}c_{k_1,\dots,k_m} f_1^{d_1-k_1}\tau(f_1)^{k_1}\dots  f_m^{d_m-k_m}\tau(f_m)^{k_m}$$
is harmonic if and only if
\begin{multline}\label{lin-sys-two}
-2\,\mu\,c_{k_1,\dots,k_m}((\sum_{i=1}^mk_i)^2-\sum_{i=1}^mk_i)=\\\sum_{j=1}^mc_{k_1,\dots k_{j-1},k_j-1,k_{j+1},\dots,k_m}(d_j+1-k_j)(\sum_{i=1}^m(d_i+k_i)-1)
\end{multline}
holds for all \,$0\leq k_i\leq d_i$.
We thus obtain an $m$-parameter family of harmonic functions.
\end{theorem}
\begin{proof}
Below we make use of the short hand notation
\begin{equation*}
F_i=f_i^{d_i-k_i}\tau(f_i)^{k_i}.
\end{equation*}
Therefore we have
$$F=\sum_{k_1=0}^{d_1}\dots\sum_{k_m=0}^{d_m}c_{k_1,\dots,k_m}F_1\dots F_m.$$
Multiple use of equations (\ref{equation-basic}) and (\ref{kappa-basic}) thus
	yields
	\begin{equation}
	\label{firstd}
	\begin{aligned}
	\tau(F)=\sum_{k_1=0}^{d_1}\dots\sum_{k_m=0}^{d_m}c_{k_1,\dots,k_m}&\big(\sum_{j=1}^m\tau(F_j)\prod_{i=1,i\neq j}^mF_i\\&+2\sum_{i<j}\kappa(F_i,F_j)\prod_{\ell=1,\ell\neq i, \ell\neq j}^mF_{\ell}\big).
	\end{aligned}
	\end{equation}
	Using Lemma\,\ref{basic-identities} we obtain
	\begin{equation*}
	\begin{aligned}
	\tau(F_i)=(d_i^2-k_i^2){f_i}^{d_i-k_i-1}\tau(f_i)^{k_i+1}+2\mu\, k_i(k_i-1)f_i^{d_i-k_i}\tau(f_i)^{k_i}
	\end{aligned}
	\end{equation*}
	and
		\begin{equation*}
	\begin{aligned}
	\kappa(F_i,F_j)=&2\mu\,k_ik_jf_i^{d_i-k_i}\tau(f_i)^{k_i}f_j^{d_j-k_j}\tau(f_j)^{k_j}\\&+\tfrac{1}{2}(d_j-k_j)(d_i+k_i)f_i^{d_i-k_i}\tau(f_i)^{k_i}f_j^{d_j-k_j-1}\tau(f_j)^{k_j+1}\\&+\tfrac{1}{2}(d_i-k_i)(d_j+k_j)f_j^{d_j-k_j}\tau(f_j)^{k_j}f_i^{d_i-k_i-1}\tau(f_i)^{k_i+1}.
	\end{aligned}
	\end{equation*}
	Plugging these results into Equation\,\ref{firstd} and comparing coefficients yields the linear system
	\begin{equation*}
		\begin{aligned}
			&\sum_{j=1}^m\big(c_{k_1,\dots,k_{j-1},k_j-1,k_{j+1},\dots,k_m}(d_j^2-(k_j-1)^2)-2c_{k_1,\dots,k_m}k_j(k_j-1)\big)\\&
			+\sum_{i<j}\big(c_{k_1,\dots,k_{j-1},k_j-1,k_{j+1},\dots,k_m}(d_i+k_i)(d_j+1-m_j)/2\\
			&\hspace{1cm}+c_{k_1,\dots,k_{i-1},k_i-1,k_{i+1},\dots,k_m}(d_i+1-k_i)(d_j+m_j)/2\\
			&\hspace{1cm}+2k_ik_j\,\mu\,c_{k_1,\dots,k_m}\big)=0.
		\end{aligned}
	\end{equation*}
	One easily verifies that for $m=1$ this linear system coincides with the system given by (\ref{lin-sys-two}) for that special case.
	An induction argument then establishes the first part of the claim.
	
	\smallskip
	
	Finally, observe that exactly the coefficients $c_{k_1,\dots,k_m}$ with $$(k_1,\dots,k_m)=(0,\dots,0,1,0,\dots,0)$$
	determine all remaining coefficients $c_{k_1,\dots,k_m}$.
	There  are exactly $m$ such coefficients, which establishes the claim.
\end{proof}

In the same vain as in the proceeding theorem we will now examine multi-homogeneous polynomials for biharmonicity.
The idea is to rewrite $\tau(F)$ such that it has the same structure as $F$.
When applying $\tau$ to $\tau(F)$ we can thus make use of the considerations contained in the proof of Theorem\,\ref{general-harmonic}.

\begin{theorem}
Let $1\leq m\leq N$ be given and $f_i=P_i/Q$, $i=1,\dots, m$ be proper biharmonic functions. The function 
	$$F=\sum_{k_1=0}^{d_1}\dots\sum_{k_m=0}^{d_m}c_{k_1,\dots,k_m} f_1^{d_1-k_1}\tau(f_1)^{k_1}\dots  f_m^{d_m-k_m}\tau(f_m)^{k_m}$$
	is proper biharmonic if and only if
	\begin{equation}
	\label{lin-sys-2}
	\begin{aligned}
	-2\,\mu\,&\tilde{c}_{k_1,\dots,k_m}((\sum_{i=1}^mk_i)^2-\sum_{i=1}^mk_i)=\\&\sum_{j=1}^m\tilde{c}_{k_1,\dots k_{j-1},k_j-1,k_{j+1},\dots,k_m}(d_j+1-k_j)(\sum_{i=1}^m(d_i+k_i)-1)
	\end{aligned}
	\end{equation}
	holds for all $0\leq k_i\leq d_i+1$, $i\in\{1,\dots,m\}$,
	where
	\begin{equation*}
		\begin{aligned}
			\tilde{c}_{k_1,\dots,k_m}:=&2\,\mu\,c_{k_1,\dots,k_m}(\sum_{j=1}^mk_j(k_j-1)+2\sum_{i<j}k_jk_i)\\
			&+\sum_{j=1}^mc_{k_1,\dots,k_{j-1},k_j-1,k_{j+1},\dots,k_m}(d_j^2-(k_j-1)^2)\\
			&+\sum_{i<j}c_{k_1,\dots,k_{i-1},k_i-1,k_{i+1},\dots,k_m}(d_i+1-k_i)(d_j+k_j)\\
			&+\sum_{i<j}c_{k_1,\dots,k_{j-1},k_j-1,k_{j+1},\dots,k_m}(d_j+1-k_j)(d_i+k_i).
		\end{aligned}
	\end{equation*}
	We thus obtain a $m$-parameter family of biharmonic functions.
\end{theorem}
\begin{proof}
As mentioned above, we first rewrite $\tau(F)$ such that it has the same structure as $F$.

\smallskip

	From (\ref{firstd}) we get
	\begin{equation*}
		\begin{aligned}
			\tau(&F)=
			\sum_{k_1=0}^{d_1+1}\dots\sum_{k_m=0}^{d_m+1}\big(2c_{k_1,\dots,k_m}\mu(\sum_{j=1}^mk_j(k_j-1)+2\sum_{i<j}k_jk_i)\\
			&+\sum_{j=1}^mc_{k_1,\dots,k_{j-1},k_j-1,k_{j+1},\dots,k_m}(d_j^2-(k_j-1)^2)\\
			&+\sum_{i<j}c_{k_1,\dots,k_{i-1},k_i-1,k_{i+1},\dots,k_m}(d_i+1-k_i)(d_j+k_j)\\
			&+\sum_{i<j}c_{k_1,\dots,k_{j-1},k_j-1,k_{j+1},\dots,k_m}(d_j+1-k_j)(d_i+k_i)\big)\Pi_{\ell=1}^mF_{\ell}.
		\end{aligned}
	\end{equation*}
	Here we make use of the convention that $c_{k_1,\dots,k_m}=0$ if either one
	of the indices is less than $0$ or
	there exists an $i\in\{1,\dots,m\}$
	such that $k_i>d_i$.
	Thus we have
	\begin{equation*}
		\begin{aligned}
			\tau(&F)=\sum_{k_1=0}^{d_1+1}\dots\sum_{k_m=0}^{d_m+1}\tilde{c}_{k_1,\dots,k_m}\Pi_{\ell=1}^mF_{\ell}.
		\end{aligned}
	\end{equation*}
	By Theorem\,\ref{general-harmonic} the identity $\tau^2(F)=0$ is therefore satisfied if and only if
	\begin{equation}
	\label{lin-sys}
	\begin{aligned}
	-2\mu&\tilde{c}_{k_1,\dots,k_m}((\sum_{i=1}^mk_i)^2-\sum_{i=1}^mk_i)=\\&\sum_{j=1}^m\tilde{c}_{k_1,\dots k_{j-1},k_j-1,k_{j+1},\dots,k_m}(d_j+1-k_j)(\sum_{i=1}^m(d_i+k_i)-1)
	\end{aligned}
	\end{equation}
	holds for all $0\leq k_i\leq d_i+1$, $i\in\{1,\dots,m\}$.
	
	\smallskip
	
	Finally, observe that exactly the coefficients $\tilde{c}_{k_1,\dots,k_m}$ with $$(k_1,\dots,k_m)=(0,\dots,0,1,0,\dots,0)$$ and $(k_1,\dots,k_m)=(0,\dots,0)$
	determine all remaining coefficients $\tilde{c}_{k_1,\dots,k_m}$.
	These in turn are determined by $c_{k_1,\dots,k_m}$ with $(k_1,\dots,k_m)=(0,\dots,0,1,0,\dots,0)$
	and $(k_1,\dots,k_m)=(0,\dots,0)$.
	We thus obtain a $m+1$-parameter family of biharmonic maps.
\end{proof}

The proof of the preceding theorem implies that for each choice of
$p,q\in\cn^n$ and $m,d_1,\dots,d_m\in\mathbb{N}$ 
there is essentially just one biharmonic map.

\begin{corollary}\label{main-cor}
	Let $p,q\in\cn^n$, $1\leq m\leq N$ and $d_1,\dots,d_m\in\mathbb{N}$ be given.
	The above construction yields -- up to scaling -- one proper biharmonic function 
	of the form 
	\begin{equation*}
\Phi_{d_1,\dots,d_m}(f_1,\dots,f_m)=\sum_{k_1=0}^{d_1}\dots\sum_{k_m=0}^{d_m}c_{k_1,\dots,k_m}f_1^{d_1-k_1}\tau(f_1)^{k_1}\dots f_m^{d_m-k_m}\tau(f_m)^{k_m},
\end{equation*}
	defined on a dense subset of $G$.
\end{corollary}

This corollary completes the construction of biharmonic functions.

\smallskip

In what follows we apply these results to finish the construction of biharmonic multi-homogeneous polynomials on $\SU n$ which we started in Section\,\ref{section-new-U(n)-II}.
In order to accomplish this we need to find eigenfunctions
$P_j,Q,R,S_j:\GLC n\to\cn$ of the Laplace-Beltrami operator $\tau$ with the same eigenvalue $\lambda$ which satisfy the conditions (\ref{PP-QQ}).
Recall that in Section\,\ref{section-new-U(n)-II}
we have chosen the 
$P_j,Q:\U n\to\cn$ to be 
$$
P_j=\sum_{k=1}^np_k a_{j}z_{kj}\ \ \text{and}\ \  Q:=Q_\beta=\sum_{k=1}^nq_k b_{\beta}z_{k\beta}.
$$
We introduce $R,S_j:\U n\to\cn$ by
$$
R=\sum_{k=1}^np_k b_{\beta}z_{k\beta}\ \ \text{and}\ \ S_j=\sum_{k=1}^nq_k a_{j}z_{kj}.
$$
By straightforward computations which make use of Equations\,\ref{equation-basic} and \ref{kappa-basic} as well as Lemma\,\ref{lemma-U(n)},
it follows that this set of functions satisfies the conditions (\ref{PP-QQ}) with $\mu=-1$.

\smallskip

Note that according to Theorem \ref {theorem-SU(n)} the function $f_{i}=P_i/Q$ is harmonic if $i=\beta$ and proper biharmonic otherwise.
Hence there exist $n-1$ proper biharmonic functions $f_i$. Thus in the above considerations we have $N=n-1$.
Consequently, Corollary\,\ref{main-cor} implies Theorem\,\ref{main-result} for $G=\SU n$.

\smallskip

In the following two sections, Sections\,\ref{section-Sp(n)} and \ref{section-SO(n)}, we will use the results of the present section to construct biharmonic 
functions on $\Sp n$ and $\SO n$, respectively.


\section{Biharmonic Functions on $\Sp n$}\label{section-Sp(n)}

In this section we show that the quaternionic unitary group $\Sp n$ falls into the general scheme that we have developed.  This can be applied to construct complex-valued proper biharmonic functions on open and dense subsets of $\Sp n$.  They are quotients of homogeneous polynomials in the matrix coefficients of the standard irreducible complex representation of $\Sp n$.

\smallskip

The quaternionic unitary group $\Sp n$ is a compact subgroup of $\U {2n}$. It is the intersection of $\U{2n}$ and the standard complex representation of the quaternionic general linear group $\GLH n$ in $\cn^{2n\times 2n}$ with 
$$
(z+jw)\mapsto q=\begin{bmatrix}z & w \\ -\bar w & \bar
z\end{bmatrix}.
$$
For the standard complex representation of the Lie algebra $\sp n$ of $\Sp n$ we have 
$$
\sp{n}=\bigg\{\begin{bmatrix} Z & W
\\ -\bar W & \bar Z\end{bmatrix}\in\cn^{2n\times 2n}
\ |\ Z^*+Z=0,\ W^t-W=0\bigg\}.
$$
The canonical orthonormal basis $\B$ for $\sp n$ is the union of the
following three sets

$$
\bigg\{\frac 1{\sqrt 2}
\begin{bmatrix}Y_{rs} & 0 \\
0 & Y_{rs}\end{bmatrix},
\frac 1{\sqrt 2}
\begin{bmatrix} iX_{rs} & 0 \\
0 & -iX_{rs} 
\end{bmatrix}\ |\ 1\le r<s\le n\bigg\},
$$

$$
\bigg\{
\frac 1{\sqrt 2}
\begin{bmatrix}
0 & X_{rs} \\
-X_{rs} & 0
\end{bmatrix}.
\frac 1{\sqrt 2}
\begin{bmatrix}
0 & iX_{rs} \\
iX_{rs} & 0
\end{bmatrix}
\ |\ 1\le r<s\le n\bigg\},
$$

$$
\bigg\{
\frac 1{\sqrt 2}
\begin{bmatrix}
0 & D_{r}  \\
-D_{r} & 0
\end{bmatrix},
\frac 1{\sqrt 2}
\begin{bmatrix}
0 & iD_{r} \\
iD_{r} & 0
\end{bmatrix},
\frac 1{\sqrt 2}
\begin{bmatrix}
iD_{r} & 0 \\
0 & -iD_{r}
\end{bmatrix}
\ |\ 1\le r\le n\bigg\}.
$$
\medskip

The following fundamental result can be found in Lemma 6.1 of \cite{Gud-Mon-Rat-1}.  It describes the behaviour of the tension field $\tau$ and the conformality operator $\kappa$ on the quaternionic unitary group $\Sp n$.  

\begin{lemma}\label{lemma-basic-Sp(n)}
For $1\le j,\alpha\le n$, let $z_{j\alpha},w_{j\alpha}:\Sp
n\to\cn$ be the matrix coefficients from the standard complex irreducible  representation of $\Sp n$. Then the following relations hold
$$
\tau(z_{j\alpha})= -\frac{2n+1}2\cdot z_{j\alpha},\ \ 
\tau(w_{j\alpha})= -\frac{2n+1}2\cdot w_{j\alpha},
$$
$$
\kappa(z_{j\alpha},z_{k\beta})=-\frac 12\cdot
z_{k\alpha}z_{j\beta},\ \ \kappa(w_{j\alpha},w_{k\beta})=-\frac
12\cdot w_{k\alpha}w_{j\beta},
$$
$$
\kappa(z_{j\alpha},w_{k\beta})=-\frac 12\cdot z_{k\alpha}w_{j\beta}.
$$
\end{lemma}

The statement of the next result is a direct consequence of Lemma \ref{lemma-basic-Sp(n)}.

\begin{lemma}
Let $M_Q$ be the following non-zero complex matrix
\begin{equation*}
M_Q=
\begin{bmatrix}
q_{11} & \cdots & q_{1,2n} \\
\vdots & \vdots & \vdots \\
q_{n1} & \cdots & q_{n,2n}
\end{bmatrix}\ \
\end{equation*}
and $Q:\Sp n\to\cn$ be the polynomial function on the quaternionic unitary group given by
$$Q(z,w)=\sum_{k,\beta}(q_{k\beta}z_{k\beta}+q_{k,n+\beta}w_{k\beta}).
$$
Then the equation $Q^2+2\kappa(Q,Q)=0$ is fulfilled if and only if the columns of $M_Q$ are pairwise linearly dependent.
\end{lemma}

\begin{proof}
The proof is similar to that of Lemma  \ref{lemma-columns-dependent}
\end{proof}

With this at hand, we can now prove the following generalisation of Theorem 6.2 in \cite{Gud-Mon-Rat-1}.

\begin{theorem}\label{theorem-Sp(n)}
Let $a,q\in\cn^{2n}$ be two non-zero vectors, $M_P$ be the following non-zero complex matrix
\begin{equation*}
M_P=
\begin{bmatrix}
p_{11} & \cdots & p_{1,2n} \\
\vdots & \ddots & \vdots   \\
p_{n1} & \cdots & p_{n,2n}
\end{bmatrix}
\end{equation*}
and the functions $P,Q:\Sp n\to\cn$ be given by
$$
P(z,w)=\sum_{j,\alpha}(p_{j\alpha}z_{j\alpha}+p_{j,n+\alpha}w_{j\alpha})
$$
and
$$Q(z,w)=\sum_{k,\beta}(q_{k}a_\beta z_{k\beta}+q_{k+n}a_{n+\beta} w_{k\beta}).
$$	
Further we define the rational function $f=P/Q$ on the open and dense subset $\{(z,w)\in\Sp n|\ Q(z,w)\neq 0\}$ of $\Sp n$. Then we have the following.
\begin{enumerate}
\item[(1)] The function $f$ is harmonic if and only if $PQ+2\,\kappa(P,Q)=0$.  This is equivalent to (i) the vector $q$ and each column vector of the matrix $M_P$ are linearly dependent or (ii) the vector $a$ and the matrix $M_P$ are of the following special form
$$a=[0,\dots,0,a_{\beta_0},0,\dots,0],$$
\begin{equation*}
M_P=
\begin{bmatrix}
0 & \cdots & 0      & p_{1\beta_0}  & 0      & \cdots & 0 \\
\vdots & \vdots & \vdots & \vdots & \vdots & \vdots & 0 \\
0 & \cdots & 0      & p_{n\beta_0}  & 0      & \cdots & 0 \\
\end{bmatrix}.
\end{equation*}
\item[(2)] The function $f$ is proper biharmonic if and only if
$PQ+2\,\kappa(P,Q)\neq 0$ i.e. if and only if neither (i) nor (ii) of (1) is satisfied.  
\end{enumerate}
\end{theorem}

\begin{proof}
The statement can be proven by exactly the same arguments as that of Theorem \ref{theorem-SU(n)}.
\end{proof}

In Section\,\ref{the-general-case} we have developed a general scheme for producing complex-valued proper biharmonic functions on certain compact subgroups of the general linear group $\GLC n$.  The next result shows that the quaternionic unitary group $\Sp n$ falls into this scheme.

\begin{lemma}\label{lemma-big-Sp(n)}
Let $a,b,p,q\in\cn^n$ be non-zero elements. Further let the polynomial functions  $P,Q,R,S:\Sp n\to\cn$ satisfying 
$$
PQ+2\,\kappa(P,Q)\neq 0
$$ 
be chosen by one of (\ref{Sp(n)-choice-1}), (\ref{Sp(n)-choice-2}) or (\ref{Sp(n)-choice-3}):	
\begin{equation}	
\begin{aligned}\label{Sp(n)-choice-1}
&
P(z)=\sum_{j}p_{j}a_\alpha z_{j\alpha},\ \ 
Q(z)=\sum_{k}q_{k}b_\beta z_{k\beta},\\	
&
R(z)=\sum_{j}p_{j}b_\beta z_{j\beta},\ \ 
S(z)=\sum_{k}q_{k}a_\alpha z_{k\alpha};
\end{aligned}
\end{equation}
\begin{equation}	
\begin{aligned}\label{Sp(n)-choice-2}
&
P(z)=\sum_{j}p_{j}a_\alpha w_{j\alpha},\ \ 
Q(z)=\sum_{k}q_{k}b_\beta z_{k\beta},\\	
&
R(z)=\sum_{j}p_{j}b_\beta z_{j\beta},\ \ 
S(z)=\sum_{k}q_{k}a_\alpha w_{k\alpha};	
\end{aligned}
\end{equation}
\begin{equation}	
\begin{aligned}\label{Sp(n)-choice-3}
&
P(z)=\sum_{j}p_{j}a_\alpha w_{j\alpha},\ \ 
Q(z)=\sum_{k}q_{k}b_\beta w_{k\beta},\\	
&
R(z)=\sum_{j}p_{j}b_\beta w_{j\beta},\ \ 
S(z)=\sum_{k}q_{k}a_\alpha w_{k\alpha}.		
\end{aligned}
\end{equation}
Then the rational quotient $f=P/Q$ satisfies the conditions given by the equations (\ref{PP-QQ}) with $\mu=-1/2$.
\end{lemma}

\begin{proof}
The statement is easily proven by exploiting Lemma \ref{lemma-basic-Sp(n)}.
\end{proof}

The result of Corollary\,\ref{main-cor} implies that of  Theorem\,\ref{main-result} in the case of $G=\Sp n$.


\section{Biharmonic Functions on $\SO n$}\label{section-SO(n)}

In this section we show that the special orthogonal group $\SO n$ falls into the general scheme that we have developed.  This can be applied to construct complex-valued proper biharmonic functions on open and dense subsets of $\SO n$.  They are quotients of homogeneous polynomials in the matrix coefficients of the standard irreducible representation of $\SO n$. 

\medskip

The special orthogonal group $\SO n$ is the compact subgroup of the real general linear group $\GLR n$ given by
$$
\SO{n}=\{x\in\GLR{n}\ |\ x\cdot x^t=I_n,\ \det x =1\}.
$$
The standard representation of the Lie algebra $\so n$ of $\SO n$ is given by the set of skew-symmetric matrices
$$
\so n=\{Y\in\glr n |\ Y+Y^t=0\}
$$
and for this we have the canonical orthonormal basis
$$
\B=\{Y_{rs}|\ 1\le r<s\le n\}.
$$

The following result was established in \cite{Gud-Sak-1}.  It describes the behaviour of the tension field $\tau$ and the conformality operator $\kappa$ on the special orthogonal group $\SO n$. 

\begin{lemma}\label{lemma-basic-SO(n)}
For $1\le j,\alpha\le n$, let $x_{j\alpha}:\SO n\to\rn$ be the real-valued matrix coefficients of the standard representation of $\SO n$.  Then the following relations hold
$$
\tau(x_{j\alpha})=-\frac {(n-1)}2\cdot x_{j\alpha},
$$
$$
\kappa(x_{j\alpha},x_{k\beta})=-\frac 12\cdot (x_{k\alpha}x_{j\beta}-\delta_{kj}\delta_{\alpha\beta}).
$$
\end{lemma}

\smallskip

As an immediate comsequence of Lemma \ref{lemma-basic-SO(n)} we have the next useful result.

\begin{lemma}
Let $M_Q$ be the following non-zero complex matrix
\begin{equation*}
M_Q=
\begin{bmatrix}
q_{11} & \cdots & q_{1n} \\
\vdots & \vdots & \vdots \\
q_{n1} & \cdots & q_{nn}
\end{bmatrix}
\end{equation*}
and $Q:\SO n\to\cn$ be the complex-valued polynomial function on the special orthogonal group given by
$$
Q(x)=\sum_{k,\alpha}q_{k\alpha}x_{k\alpha}.
$$
Then the equation $Q^2+2\kappa(Q,Q)=0$ is fulfilled if and only if the columns of $M_Q$ are isotropic and pairwise linearly dependent.
\end{lemma}

\begin{proof}
The statement follows easily from the fact that
$$
Q^2+2\kappa(Q,Q)=\sum_{j,k,\alpha,\beta}
(q_{j\alpha}q_{k\beta}-q_{k\alpha}q_{j\beta})	x_{j\alpha}x_{k\beta}+\sum_{j,\alpha}q_{j\alpha}^2=0.
$$
\end{proof}

The next theorem shows how the standard representation of the special orthogonal group $\SO n$ can be employed to construct proper biharmonic functions.  It generalises the result of Theorem 5.2 of \cite{Gud-Mon-Rat-1}.

\begin{theorem}\label{theorem-SO(n)}
Let $a,q\in\cn^n$ be two non-zero vectors such that $(a,a )\neq 0$, $M_P$ the following non-zero complex matrix
\begin{equation*}
	M_P=
	\begin{bmatrix}
		p_{11} & \cdots & p_{1n}\\
		\vdots & \ddots & \vdots\\
		p_{n1} & \cdots & p_{nn}
	\end{bmatrix}
\end{equation*}
and the polynomial functions $P,Q:\SO n\to\cn$ be given by
$$P(x)=\sum_{j,\alpha}p_{j\alpha}x_{j\alpha}\ \ \text{and}\ \ Q(x)=\sum_{k,\beta}q_{k}a_\beta x_{k\beta}.$$
Further we define the rational function $f=P/Q$ on the open and dense subset $\{x\in\SO n|\ Q(x)\neq 0\}$ of $\SO n$.  Then we have the following.
\begin{enumerate}
\item[(1)] 
The function $f$ is harmonic if and only if $PQ+2\,\kappa(P,Q)=0$.  This is equivalent to (i) the vector $q$ and each column vector of the matrix $M_P$ are linearly dependent or (ii) the vector $a$ and the matrix $M_P$ are of the following special form
$$a=[0,\dots,0,a_{\beta_0},0,\dots,0],$$
\begin{equation*}
M_P=
\begin{bmatrix}
0 & \cdots & 0      & p_{1\beta_0}  & 0      & \cdots & 0 \\
\vdots & \vdots & \vdots & \vdots & \vdots & \vdots & 0 \\
0 & \cdots & 0      & p_{n\beta_0}  & 0      & \cdots & 0 \\
\end{bmatrix}.
\end{equation*}
\item[(2)] 
The function $f$ is proper biharmonic if and only if
	$PQ+2\,\kappa(P,Q)\neq 0$ i.e. if and only if neither (i) nor (ii) of (1) is satisfied.  
\end{enumerate}	
\end{theorem}

\begin{proof}
The statement can be proven by exactly the same arguments as that of Theorem \ref{theorem-SU(n)}.
\end{proof}

In Section\,\ref{the-general-case} we have developed a general scheme for producing complex-valued proper biharmonic functions on certain compact subgroups of the general linear group $\GLC n$.  The next result shows that the special orthogonal group $\SO n$ fits into this scheme.

\begin{lemma}\label{lemma-big-SO(n)}
Let $n\ge 4$ and $a,b,p,q\in\cn^n$ be non-zero elements such that either $(a,a)=(b,b)=(a,b)=0$ or  $(p,p)=(p,q)=(q,q)=0$.  Further let the polynomial functions $P,Q,R,S:\SO n\to\cn$ satisfying 
$$
PQ+2\,\kappa(P,Q)\neq 0
$$ 
be given by
\begin{equation*}
\begin{aligned}
&
P(x)=\sum_{j,\alpha}p_{j}a_\alpha x_{j\alpha},\ \ 
Q(x)=\sum_{k,\beta }q_{k}b_\beta x_{k\beta },\\
&
R(x)=\sum_{j,\beta }p_{j} b_\beta x_{j\beta },\ \ 
S(x)=\sum_{k,\alpha}q_{k}a_\alpha x_{k\alpha}.
\end{aligned}
\end{equation*}
Then the rational quotient $f=P/Q$ satisfies the conditions given by the equations (\ref{PP-QQ}) with $\mu=-1/2$.
\end{lemma}

\begin{proof}
	The statement is easily proven by exploiting Lemma \ref{lemma-basic-SO(n)}.
\end{proof}

The result of Corollary\,\ref{main-cor} implies that of  Theorem\,\ref{main-result} in the case of $G=\SO n$.


\section{Harmonic Morphisms}\label{harmonic-morphisms}

In this section we manufacture new eigenfamilies of complex-valued functions defined on open and dense subsets of the special unitary group $\SU n$, $\Sp n$ and $\SO n$. Their elements are ingredients for a recipe of harmonic morphisms, as we now will describe. 
We just carry out the considerations for the group $\U n$, the proof for the other two cases are the same.

\medskip

The following notion of an eigenfamily was introduced in the paper \cite{Gud-Sak-1}.

\begin{definition}\label{defi:eigen}
Let $(M,g)$ be a Riemannian manifold.  Then a set
$$\E =\{\phi_i:M\to\cn\ |\ i\in I\}$$ of complex-valued functions is called an {\it eigenfamily} on $M$ if there exist complex numbers $\lambda,\mu\in\cn$ such that
$$
\tau(\phi)=\lambda\,\phi\ \ \text{and}\ \ \kappa(\phi,\psi)=\mu\,\phi\,\psi,
$$ 
for all $\phi,\psi\in\E$. A set
$$
\Omega =\{\phi_i:M\to\cn\ |\ i\in I\}
$$ 
is called an {\it orthogonal harmonic family} on $M$ if  for all $\phi,\psi\in\Omega$ $$\tau(\phi)=0\ \ \text{and}\ \ \kappa(\phi,\psi)=0.$$
\end{definition}

The reader should note that that every element of an orthogonal harmonic family is a harmonic morphism since it is both harmonic and horizontally conformal.

\smallskip

Below let $\beta\in\{1,\dots,n\}$ be fixed.
Let $P_j,Q:\U n\to\cn$  be given as in Section\,\ref{section-new-U(n)-II}, that is
$$
P_j=\sum_{k=1}^np_k a_{j}z_{kj}\ \ \text{and}\ \  Q:=Q_\beta=\sum_{k=1}^nq_k b_{\beta}z_{k\beta}.
$$
Further let $f_{j}:W\to\cn$ be the quotient $f_{j}=P_j/Q$ defined on the open and dense subset
$$
W=\{z\in\U n|\ Q(z)\neq 0\}
$$
of the unitary group $\U n$. According to Theorem \ref {theorem-SU(n)} the functions $f_{j}$ are harmonic and proper biharmonic if and only if $j\neq\beta$.

\begin{proposition}\label{proposition-eigenfamily-U(n)}
Let $k\in\mathbb{N}^+$ and $\E_k$ be the following set of complex-valued functions, defined on the open and dense subset $W$ of $\U n$,
$$
\E_k=\{\tau(f_{j})^k:W\to\cn|\ \alpha\neq\beta\}.
$$
Then $\E_k$ is an eigenfamily on $W$.  The corresponding statement holds for the induced family on the special unitary group $\SU n$.
\end{proposition}
\begin{proof}
This is an immediate consequence of Lemma\,\ref{basic-identities}.
\end{proof}

The next result shows how the eigenfamily $\E$ in Proposition  \ref{proposition-eigenfamily-U(n)} produces a large collection of harmonic morphisms on the special unitary group $\SU n$.

\begin{theorem}\cite{Gud-Sak-1}\label{theo:rational}
Let $(M,g)$ be a Riemannian manifold and 
$$
\E =\{\phi_1,\dots,\phi_n\}
$$ 
be a finite eigenfamily of complex valued functions on $M$. If $P,Q:\cn^n\to\cn$ are linearily independent homogeneous polynomials of the same positive degree then the quotient
$$
\frac{P(\phi_1,\dots ,\phi_n)}{Q(\phi_1,\dots ,\phi_n)}
$$ 
is a non-constant harmonic morphism on the open and dense subset
$$
\{p\in M| \ Q(\phi_1(p),\dots ,\phi_n(p))\neq 0\}.
$$
\end{theorem}

\smallskip

We now discuss the special case when the eigenfamily is both orthogonal and harmonic.

\begin{proposition}\label{proposition-orthogonal-family}
Let $q\in\cn^n$ be non-zero and define the complex-valued polynomial functions $Q_\alpha:\U n\to\cn$ by 
$$
Q_\alpha(z)=\sum_{j}q_{j}z_{j\alpha}.
$$
Then the collection 
$$
\Omega=\{{Q_\alpha}/{Q_\beta}|\ \alpha\neq\beta\}
$$
is a harmonic orthogonal family on the following open subset of the unitary group $\U n$
$$
W=\bigcap_{\alpha =1}^n\{z\in\U n|\ Q_\alpha (z)\neq 0\}.
$$
The corresponding statement holds for the induced family on the special unitary group $\SU n$.
\end{proposition}

\begin{proof} 
It is a direct consequence of Lemma  \ref{lemma-columns-dependent} that 
\begin{equation}\label{eqn-Q}
Q_\alpha Q_\beta+\kappa(Q_\alpha,Q_\beta)=0
\end{equation}
so according to Theorem \ref{theorem-SU(n)} all the elements of $\Omega$ are harmonic. Finally, a simple computation, repeately using equation (\ref{eqn-Q}), shows that $\kappa(f,h)=0$ for all $f,h\in\Omega$.
\end{proof}

For this special case we have the following useful result.  This tells us how the family $\Omega$ in Proposition \ref{proposition-orthogonal-family} provides a large collection of harmonic morphisms on the special unitary group $\SU n$.

\begin{theorem}\cite{Gud-8}\label{theo:orthogonal}
Let $(M,g)$ be a Riemannian manifold and
$$\Omega=\{\phi_k:M\to\cn\ |\ k=1,\dots ,n\}$$ be a finite orthogonal harmonic family on $M$.  Let $\Phi:M\to\cn^n$ be the map given by $\Phi=(\phi_1,\dots,\phi_n)$ and $U$ be an open subset of $\cn^n$ containing the image $\Phi(M)$ of $\Phi$. If $$\H=\{h_i:U\to\cn\ |\ i\in I\}$$ is a collection of holomorphic functions then 
$$
\widetilde\Omega=\{\psi:M\to\cn\ |\ \psi=h(\phi_1,\dots ,\phi_n),\ h\in\H\}
$$ is an orthogonal harmonic family on $M$.
\end{theorem}

\end{document}